\documentclass[a4paper,12pt]{amsart}

\usepackage{amsfonts}
\usepackage{amssymb}
\usepackage{amsmath}
\usepackage[mathscr]{euscript}
\usepackage{hyperref}

\makeindex

\newtheorem {cor}{\textbf Corollary}[section]
\newtheorem{teo}[cor]{\textbf Theorem}
\newtheorem{lem}[cor]{\textbf Lemma}
\newtheorem{prop}[cor]{\textbf Proposition}

\renewcommand{\leq}{\leqslant}
\renewcommand{\geq}{\geqslant}

\theoremstyle{definition}

\newcommand{\GL}[2]{{\sf GL(}#1, #2)}
\newcommand{\Aut}[1]{{\sf Aut(}#1)}
\newcommand{\Sym}[1]{{\sf Sym(}#1)}

\newcommand{\F}{\mathbb{F}}

\newcommand{\N}{\mathbb{N}}

\newcommand{\agl}[2]{{\sf AGL}(#1,#2)}

\renewcommand{\P}{\mathscr P}
\newcommand{\Cay}[2]{{\sf Cay}(#1,#2)}

\makeindex

\title[Generalized Paley graphs and synchronization]
{Cliques and colorings in generalized
Paley graphs and an approach to synchronization}

\author{Csaba Schneider and Ana Silva}

\address[Schneider]{Departamento de Matem\'atica\\
Instituto de Ciências Exatas\\
Universidade Federal de Minas Gerais\\
Av.\ Ant\^onio Carlos 6627\\31270-901\\
Belo Horizonte, MG, Brazil}
\address[Silva]{Departamento de Matemática\\
Faculdade de Ciências da
Universidade de Lisboa\\
Campo Grande,
1749-016 Lisboa,
Portugal}

\begin{document}

\begin{abstract}
Given a finite field, one can form a directed 
graph using the field elements as vertices
and connecting two vertices if their difference lies 
in a fixed subgroup of the multiplicative group.
If $-1$ is contained
in this fixed subgroup, 
then we obtain an undirected graph that is referred to as 
a generalized Paley graph.
In this paper we study generalized Paley graphs whose clique and 
chromatic numbers coincide and link this theory to the study of the 
synchronization property in 1-dimensional primitive affine permutation groups.
\end{abstract}

\maketitle

\section{Introduction}

The synchronization property for permutation groups emerged from
problems related to the \v Cern\'y conjecture~\cite{cerny} in 
the theory of finite state automata and was first 
independently defined by 
Jo\~ao Ara\'ujo and Benjamin Steinberg;
see \cite{joao,arnoldandsteinberg}.
As was shown by
Arnold and Steinberg~\cite{arnoldandsteinberg}, a synchronizing group is always primitive.
However,  the converse is not true, which makes it interesting to study
the problem: Which primitive groups are synchronizing? 
Neumann~\cite{neumann} proved several basic properties of synchronizing groups,
and he also observed that a primitive group that can be embedded into a 
wreath product in 
product action is always non-synchronizing. 
This shows, in the terminology of Cameron~\cite{cameron}, 
that a synchronizing group
must be  a basic primitive group. 
Therefore the
O'Nan-Scott type of a 
synchronizing group is almost simple, affine, or 
simple diagonal. The study of precisely which groups in these classes are
synchronizing leads to deep problems in combinatorics, graph theory, and 
finite geometry. For instance, deciding whether a classical group
in its action on the associated polar space is synchronizing requires
a study of spreads and ovoids of the polar space;
see\cite[Theorem~9 of Part~6]{talk}.

In this paper we study which of the members of a particular class of affine 
groups are synchronizing. These groups are primitive subgroups of the 
1-dimensional affine groups $\agl 1{q}$ acting on the set $\F_{q}$
for odd prime-powers $q$.
Neumann~\cite{neumann} introduced a graph theoretic characterization
of non-synchronizing groups in terms of cliques and colorings of 
their generalized undirected orbital graphs. The undirected orbital
graphs of the groups we study are isomorphic to certain graphs that can
be considered as generalizations of the Paley graphs of finite fields, first
defined in~\cite{Paley}.
Such graphs were also considered in~\cite{praeger}.
Hence, using Neumann's characterization of non-synchronizing groups
 (Lemma~\ref{synchgraphs}), the study of the synchonization
property in
such affine groups leads to an investigation of 
the clique and chromatic numbers of generalized Paley graphs.

Calculating the clique and the chromatic numbers can be 
notoriously difficult already for Paley graphs as shown by~\cite{GenPaley};
see also~\cite{MO} for a recent discussion. 
Hence our results are far from conclusive. Nevertheless, we  
prove some theorems that describe situations when the chromatic and the clique 
numbers of generalized Paley graphs are equal (Theorem~\ref{main1}), 
and we apply
these results to determine if certain 1-dimensional affine 
groups are synchronizing (Theorem~\ref{main2}).

In Section~\ref{sec:synch}, we give a brief summary of synchronizing groups. 
In Section~\ref{sec:paley}, we link affine synchronizing groups to generalized 
Paley graphs. In Section~\ref{sec:clique}, we prove our main result 
(Theorem~\ref{main1}) concerning
the clique and chromatic numbers of generalized Paley graphs that state 
several sufficient or necessary conditions for these numbers to be equal.
Finally, in Section~\ref{sec:groups}, we apply Theorem~\ref{main1} in the
study of synchronization in 1-dimensional affine groups.

\subsection*{Acknowledgment}
The research that led to this paper was carried out as part of the second 
author's MSc project in the Center of Algebra of the University of Lisbon.
We both are grateful to the Center for supporting our work.

\section{Synchronizing Groups}\label{sec:synch}


Let $\Omega$ be a finite set and let $\P$ be a partition of $\Omega$. 
A subset $S$ of
$\Omega$ is called a \emph{section} for $\mathscr P$ if $S$ contains
precisely one element from each part of $\mathscr P$. Let $G\leq \Sym
\Omega$.
If $\mathscr P$ admits a section $S$ such that $Sg$ is a
section for all $g \in G$ then $\mathscr P$ is said to be a
\emph{section-regular} partition for $G$ or a \mbox{\emph{$G$-regular}}
partition. In this case we say that the section $S$ {\em witnesses
the $G$-regularity} of $\P$.

The partitions $\{\Omega\}$ and $\{\{\omega\}\mid \omega\in\Omega\}$ are
clearly $G$-regular for any permutation group $G$ 
acting on $\Omega$. These partitions are
said to be \emph{trivial}.
A permutation group $G$ acting on a set $\Omega$ is called
\emph{synchronizing} if $G\neq1$ and there are no non-trivial
$G$-regular partitions of $\Omega$. It follows that synchronizing
groups are transitive and primitive, since a $G$-invariant partition is
clearly $G$-regular.  Furthermore, 
it is not hard to see that 
2-homogeneous groups are synchronizing.

The following lemma summarizes the basic properties of $G$-regular 
partitions for transitive groups; see~\cite[Section~2]{neumann}.

\begin{lem}\label{synchlemma}
Let $G$ be a transitive group acting on a set $\Omega$. 
\begin{enumerate}
\item A $G$-regular partition 
is uniform, that is, all its parts have the same size.
\item If $G$ is a primitive group and $\P$ is a non-trivial 
$G$-regular 
partition, then $|\P|\geq 3$ and $|P|\geq 3$ for all $P\in\P$.
\end{enumerate}
\end{lem}

As we mentioned in the introduction, Neumann~\cite{neumann} showed that
a primitive group that can be embedded into a wreath product in product
action is non-synchronizing. In the following lemma we rephrase this
result in the language of cartesian decompositions, introduced in~\cite{csaba}.
A \emph{cartesian decomposition}\index{cartesian decomposition} of a set $\Omega$ is a set $\Sigma =\{$$\mathscr P$$_1, \ldots , $$\mathscr P$$_t\}$ of \mbox{non-trivial} partitions of $\Omega$ such that
\begin{center}
$|P_1 \cap \cdots \cap P_t| =1$ for all $P_1 \in $ $\mathscr P$$_1, \ldots, P_t \in $ $\mathscr P$$_t$.
\end{center}

\begin{lem}\label{cartdecomp}
Let $G$ be a primitive group acting on a finite set $\Omega$. 
If $G$ preserves a cartesian decomposition of $\Omega$, 
then $G$ is non-synchronizing.
\end{lem} 
\begin{proof}
Let $\Sigma=\{$$\mathscr P$$_1, \ldots , $$\mathscr P$$_t\}$ be a cartesian decomposition of $\Omega$ which is invariant under $G$.
We start by showing that the group $G$ must be transitive on $\Sigma$.
Assume by contradiction that $G$ is intransitive on $\Sigma$. Then we may assume without loss of generality that $\Sigma_1=\{$$\mathscr P$$_1,\ldots,$$\mathscr P$$_s\}$, with $s<t$,  is a $G$-orbit. 

Let $P_i \in \mathscr P_i$, for $i \in \{1,\ldots,s\}$, and set $B=P_1\cap\cdots\cap P_s$. We claim that $B$ is a non-trivial block for $G$. Since the partitions $\mathscr P_j$ are non-trivial for all $j \in \{1,\ldots,t\}$, we have that $B \not= \Omega$.
Next, we show that  $|B|>1$.  Let $P_r, P_r' \in $ $\mathscr P$$_r$ with $r>s$ and such that $P_r\not= P_r'$. Then, since $\Sigma$ is a cartesian decomposition of $\Omega$, it follows that $B\cap P_r \not=\emptyset$ and $B \cap P_r'\not= \emptyset$. As $P_r\cap P_r'=\emptyset$, we obtain $B$ contains at least two elements, as was claimed.
Now let $g\in G$. If 
$\{P_1g,\ldots,P_sg\}=\{P_1,\ldots,P_s\}$, then $Bg=B$, while otherwise, 
$Bg\cap B=\emptyset$. Thus $B$ is non-trivial a block for $G$, which
is a contradiction since we assume that $G$ is primitive.
Hence $G$ is transitive on the cartesian decomposition $\Sigma$.

Next, we  show that $\P_1$ is a $G$-regular partition of $\Omega$. 
By the first part of the proof, $G$ is transitive on $\Sigma$.
Then $|$$\mathscr P$$_i|=|$$\mathscr P$$_j|$ for all $\mathscr P_i,\ \mathscr P_j \in \Sigma$.
For all $i \in \{1,\ldots,t\}$, fix a bijection $\alpha_i : $$\mathscr P$$_1 \rightarrow $$\mathscr P$$_i$. 
Define, for $P\in\mathscr P_1$ the element 
$\omega_P\in\Omega$ to be the unique element in $P\cap P\alpha_2\cap\cdots\cap P\alpha_t$. Set $S=\{\omega_P : P \in $ $\mathscr P_1\}$. Note that $S$ is a 
section for all $\mathscr P$$_i \in \Sigma$. 
Thus if  $g\in G$, then $S$ is a section for $\P_1g^{-1}$, which gives  
 that $Sg$ is a section for $\P_1$.
Therefore $S$ witnesses the section-regular property for $\mathscr P$$_1$ and thus $\mathscr P_1$ is a non-trivial $G$-regular partition.
Hence $G$ is a non-synchronizing group.
\end{proof}

Neumann introduced in~\cite{neumann} a graph theoretic characterization
of synchronizing groups. If $G$ is a permutation group acting on 
$\Omega$ then the $G$-action on $\Omega$ can naturally be extended to the
sets
\begin{eqnarray*}
\Omega^2&=&\{(\alpha,\beta)\mid\alpha,\ \beta\in\Omega\}\mbox{ and}\\
\Omega^{\{2\}}&=&\{\{\alpha,\beta\}\mid\alpha,\ \beta\in\Omega\mbox{ with }
\alpha\neq\beta\}.
\end{eqnarray*}

Suppose that $G$ is transitive on $\Omega$.
The $G$-orbits in $\Omega^2$ are called {\em orbitals} and we will refer
to the $G$-orbits on $\Omega^{\{2\}}$ as {\em undirected orbitals}. 
For any subset $\Delta$ of $\Omega^2$, the pair $(\Omega,\Delta)$ is a
directed graph, while for such a subset $\Delta$ in $\Omega^{\{2\}}$, 
the pair $(\Omega,\Delta)$ is an undirected simple graph. 
In this paper the word `graph' will mean an underected simple graph.
If we refer to a directed graph, then we will write `directed graph'.
Such a graph or directed graph
admits $G$ as a group of automorphisms if and only if $\Delta$ is 
$G$-invariant; that is, it is a union of orbitals or undirected orbitals.

A complete subgraph of a graph $\Gamma$ is called a {\em clique}. The largest 
among the cardinalities of the cliques 
of a graph is said to be the {\em clique number}. An empty 
subgraph of $\Gamma$ is said to be an {\em independent set} and the largest
among the 
cardinalities of the independent sets is 
the {\em independence number}.
The 
{\em chromatic number} is the smallest number of colors necessary 
to color $\Gamma$ such that adjacent vertices receive
different colors. 

\begin{lem}[Theorem 10 of Chapter 5 \cite{talk}]\label{synchgraphs}
A transitive group $G$ acting on $\Omega$ is non-synchronizing
if and only if there is a $G$-invariant subset $\Delta$ of $\Omega^{\{2\}}$ 
such that the chromatic number of the 
undirected graph $(\Omega,\Delta)$ is equal to its clique number.
\end{lem}

We will characterize the undirected orbitals of 1-dimensional
affine groups in Lemma~\ref{UndirectedOrbits} using 
the following connection between orbitals and suborbits; 
see~\cite[Section~3.2]{dixonandmortimer}.
Suppose that $G$ is a transitive group acting on $\Omega$.
For a fixed $\alpha\in\Omega$, the correspondence
$\Delta\mapsto\Delta(\alpha)$ where 
$\Delta(\alpha)=\{\beta\mid(\alpha,\beta)\in\Delta\}$ is a bijection
between the set of $G$-orbitals and the set of orbits of the stabilizer 
$G_\alpha$. The orbits of the stabilizer $G_\alpha$ are referred to
as {\em suborbits}. If $\Delta$ is an orbital, then so is the set
$\Delta'=\{(\beta,\gamma)\mid(\gamma,\beta)\in\Delta\}$, and $\Delta'$ is 
called the {\em paired orbital} of $\Delta$. In this case 
we will also say that the 
suborbit $\Delta'(\alpha)$ is the {\em paired suborbit} of $\Delta(\alpha)$. 
The orbital $\{(\beta,\beta)\mid\beta\in\Omega\}$ is
referred to as the {\em diagonal orbital}.
The undirected orbitals are in one-to-one correspondence with the set of
sets 
$\{\Delta,\Delta'\}$ where $\Delta$ and $\Delta'$ are
non-diagonal  paired orbitals (such a 
set can have one or two elements depending on whether $\Delta'=\Delta$), and
hence
they are also in one-to-one correspondence with
the set of sets $\{\Delta(\alpha),\Delta'(\alpha)\}$ where $\Delta(\alpha)$
and $\Delta'(\alpha)$ are paired suborbits that are 
not equal to $\{\alpha\}$.





\section{Generalized Paley graphs as undirected orbital graphs}\label{sec:paley}

The objective of this paper is to study which of the 
subgroups of the 1-dimensional affine general linear groups
are synchronizing. 
These permutation groups    can be constructed as follows. Let $q$ be 
an odd prime-power and let $m$ be a divisor of $q-1$. 
We let 
$\F_q$ denote the field of $q$ elements and also let  
$\F_q^*$ denote the multiplicative group $\F_q\setminus\{0\}$. 
Set $S_{q,m}=\{\alpha^m \mid \alpha \in \F_q^*\}$.
Then $S_{q,m}$ is a subgroup of $\F_q^*$ 
with order $(q-1)/m$.
For an
element $\alpha\in \F_q$, let $\sigma_\alpha$ denote the translation map 
$\sigma_\alpha:\beta\mapsto \beta+\alpha$ for all $\beta\in\F_q$. 
If $\alpha\in\F_q^*$, then define the multiplication 
map $\mu_\alpha$ as $\mu_\alpha:\beta\mapsto \beta\alpha$ for all $\beta\in\F_q$. 
Then $\sigma_\alpha$ and, for $\alpha\in\F_q^*$, 
$\mu_\alpha$ are elements of $\Sym {\F_q}$. We let $T$ denote the 
set of elements  $\sigma_\alpha$ with $\alpha\in\F_q$ and $Z_{q,m}$ 
denote the set of 
elements $\mu_\alpha$ with $\alpha\in S_{q,m}$. Then $T$ and $Z_{q,m}$ 
are subgroups of $\Sym{\F_q}$ and  it is well known that
$\left<T,Z_{q,m}\right>=T\rtimes Z_{q,m}$. We set $G_{q,m}$ to be 
$T\rtimes Z_{q,m}$. 
We consider $G_{q,m}$ as a permutation group of degree $q$ acting
on $\F_q$. 
Then $G_{q,1}=\agl 1q$,  and hence $G_{q,m}$ is a transitive subgroup of $\agl 1q$ 
containing $T$.

\begin{lem}
Suppose that $q$ and $m$ are as above
and suppose that $q=p^d$ where $p$ is a prime. Set $r=(q-1)/m$. Then 
$G_{q,m}$ is a primitive permutation group 
if and only if $r\nmid p^i-1$ for all 
$1\leq i\leq d-1$.
\end{lem}
\begin{proof}
The group $G_{q,m}$ is primitive if and only if the stabilizer $Z_{q,m}$ of 
0 is irreducible considered as a subgroup of 
$\GL dp$~\cite[Section~4.7]{dixonandmortimer}.
As  $|Z_{q,m}|=r$, this
happens if and only if $r\nmid p^i-1$ for all $i\in\{1,\ldots,d-1\}$
(see~\cite[Section~2.3]{short}).
\end{proof}

In order to study whether the groups $G_{q,m}$ are synchronizing,
we use Lemma~\ref{synchgraphs}, and hence we need first describe the 
undirected orbitals of $G_{q,m}$ where
$m$ is a divisor of $q-1$ as above. Set $r=(q-1)/m$ and set
\begin{equation}\label{m1}
(\bar r,\bar m)=\left\{\begin{array}{ll}
(r,m)\mbox{ if $r$ is even;}\\
(2r,m/2)\mbox{ if $r$ is odd.}
\end{array}\right.
\end{equation}
Thus $\bar r$ is always even and $\bar r\bar m=q-1$.
Let $\gamma$ be a primitive element of $\F_q$. Then we define, 
for \mbox{$i \in \{0,\ldots,\bar m-1\}$},
\begin{eqnarray}\label{UndirectedOrbitals}
\Delta_i=\left\{ \{\alpha,\beta\} \mid \alpha-\beta 
\in S_{q,\bar m}\gamma^i\right\}.
\end{eqnarray}
As $\bar r$ is even, $-1 \in S_{q,\bar m}$, and so the pairs 
$(\F_{q},\Delta_i)$ define undirected graphs with vertex sets $\F_q$. 
Let $\Gamma_i$ denote the graph $(\F_{q},\Delta_i)$.

To construct the other main object of this paper, consider
now an odd prime-power $q$ as above, and let $m$ be a natural number such that 
$2m\mid q-1$. 
We define an undirected 
graph $\Gamma_{q,m}$ as
follows.  The vertex set of $\Gamma_{q,m}$ is $\F_q$ and its edge set is
$$
\{\{\alpha,\beta\}\mid \alpha-\beta \in
S_{q,m}\}. 
$$
The fact that $2m\mid (q-1)$ guarantees that $-1\in S_{q,m}$, and so
the edge set of $\Gamma_{q,m}$ is well-defined.
The graph
$\Gamma_{q,m}$ is called a \emph{generalized Paley graph} for the field
$\F_{q}$.

\begin{lem}\label{UndirectedOrbits}
Using the notation above, the following is valid.
\begin{enumerate}
\item $\Delta_0, \cdots, \Delta_{\bar m-1}$ are precisely the undirected
$G_{q,m}$-orbitals.
\item $\Gamma_{i} \cong \Gamma_{j}$ for all $i,\ j \in \{0,\ldots,\bar m-1\}$;
\item $\Gamma_{0}=\Gamma_{q,\bar m}$.
\end{enumerate}
\end{lem}
\begin{proof}
As described after Lemma~\ref{synchgraphs}, 
there is a bijection between the set of undirected 
orbitals of $G_{q,m}$ and the
unions of paired suborbits of $G_{q,m}$ with respect 
to the element 0 whose stabilizer is $Z_{q,m}$. 
These suborbits of $G_{q,m}$ are, in addition to
the trivial suborbit $\{0\}$, precisely the cosets $S_{q,m}\gamma^i$ with
$i\in\{0,\ldots,m-1\}$. The orbital that corresponds to the suborbit
$S_{q,m}\gamma^i$ is the $G_{q,m}$-orbit of $(0,\gamma^i)$. Its paired orbital
is the $G_{q,m}$-orbit of $(\gamma^i,0)$. Now $\sigma_{-\gamma^i}\in G_{q,m}$, which
gives that the $G_{q,m}$-orbit of $(\gamma^i,0)$ is equal to the $G_{q,m}$-orbit
of $(\gamma^i,0)\sigma_{-\gamma^i}=(0,-\gamma^i)$. Thus the paired suborbit of
$S_{q,m}\gamma^i$ is $-S_{q,m}\gamma^i$. The union of $S_{q,m}\gamma^i$ and 
$-S_{q,m}\gamma^i$ is equal to $S_{q,\bar m}\gamma^i$. Thus the undirected orbitals
of $G_{q,m}$ are in one-to-one correspondence with the 
cosets $S_{q,\bar m}\gamma^i$ where $i\in\{0,\ldots,\bar m-1\}$. 
As $G_{q,m}=Z_{q,m}T$, where $T$ is the group of
translations $\sigma_\alpha$ with $\alpha\in\F_q$, 
for such a 
coset $S_{q,\bar m}\gamma^i$, the corresponding undirected orbital is
\begin{multline*}
\{\{0,\beta\}\sigma_\alpha\mid\beta\in S_{q,\bar m}\gamma^i\mbox{ and }
\alpha\in \F_q\}=\\\{\{\alpha,\beta+\alpha\} \mid \beta\in S_{q,\bar m}\gamma^i\mbox{ and }
\alpha\in \F_q\}=\\\{\{\alpha,\beta\}\mid\alpha-\beta\in S_{q,\bar m}\gamma^i\}=\Delta_i.
\end{multline*}
This shows assertion~(1).

(2) By the definition of the sets $\Delta_i$, we have that 
$\Delta_0\gamma^i=\Delta_i$. Hence $\mu_{\gamma^i}$ is an isomorphism between the graphs $\Gamma_{0}$ and $\Gamma_{i}$ for all $i \in \{0,\ldots,\bar m-1\}$.

(3) This follows from the definitions of $\Delta_0$ and the graph 
$\Gamma_{q,\bar m}$. 
\end{proof} 
For a group $G$ and a subset $C$ such that 
$C^{-1}=C$ and $1\not\in C$, the {\em Cayley graph}
$\Cay GC$ is defined as the graph with vertex set $G$ 
and in which $g$ and $h$ are connected if and only if $gh^{-1}\in C$. 
A Cayley graph $\Cay GC$ is said to be \emph{normal} if
$C$ is closed under conjugation by elements of $G$. 
Clearly, if $G$ is abelian, then
all Cayley graphs of $G$ are normal. 
A graph is said to be 
{\em regular} if every vertex has the same number of neighbours. In this case, 
the number of neighbours of a vertex 
is denoted by 
 $\partial(\Gamma)$. For a graph $\Gamma=(\Omega,E)$, we
let $\overline\Gamma$ denote the \emph{complement graph} defined by 
$(\Omega,\Omega^{\{2\}}\setminus E)$.

\begin{cor}\label{PropertiesGenPaleyGraphs}
Let $q$ be an odd prime-power and $m\geq 2$ be 
such that $2m \mid q-1$. 
\begin{enumerate}
\item The graph $\Gamma_{q,m}$ is vertex-transitive and edge transitive.
\item We have that $\Gamma_{q,m}$ is a regular graph and $\partial(\Gamma_{q,m})=(q-1)/m$.
\item $\Gamma_{q,m}$ is isomorphic to a subgraph of its complement graph $\overline{\Gamma_{q,m}}$.
\item We have
that $\Gamma_{q,m}=\Cay{\F_q}{S_{q,m}}$, viewing $\F_q$ as an additive group,  
and hence $\Gamma_{q,m}$ is a normal Cayley graph.
\end{enumerate}
\end{cor}
\begin{proof}
As, by Lemma~\ref{UndirectedOrbits}(3), $\Gamma_{q,m}=(\F_q,\Delta_0)$
and $(\F_q,\Delta_0)$ is vertex- and edge-transitive, we obtain statement~(1).
The fact that $\Gamma_{q,m}$ is regular is a consequence of~(1).
Further, by the definition of $\Gamma_{q,m}$, the set $S_{q,m}$
coincides with the set of neighbours of the vertex 0. Thus
$\partial\Gamma=|S_{q,m}|=(q-1)/m$.
To show statement~(3), 
we note that the multiplication map $\mu_\gamma$,
where $\gamma$ is a primitive element of $\F_q$, 
embeds the edge set of $\Gamma_{q,m}$ into
its complement in $\Omega^{\{2\}}$. 
For the proof of statement~(4), 
we notice that the definition of $\Gamma_{q,m}$ implies that
$\Gamma_{q,m}=\Cay{\F_q}{S_{q,m}}$, 
and so $\Gamma_{q,m}$ is a normal Cayley graph.
\end{proof}

\section{Cliques and colorings of generalized Paley graphs}\label{sec:clique}

We denote the clique number, the independence number, and the chromatic
number of a graph $\Gamma$ by $\omega(\Gamma)$, $\alpha(\Gamma)$, 
and $\chi(\Gamma)$, respectively.  
In a vertex-transitive graph $\Gamma$ with $n$ vertices, 
the existence of a clique $C$ and an 
independent set $A$ such that $|C||A|=n$ guarantees that $\omega(\Gamma)=|C|$ 
and $\alpha(\Gamma)=|A|$. We prove this in Lemma ~\ref{CAequal} and in 
the proof we use the following technical result; see~\cite[Chapter 3, 
Theorem~8]{talk}.

\begin{lem}\label{lambda}
Suppose that $G$ is a transitive group acting on $\Omega$ and 
$A$ and $B$ are subsets of $\Omega$ such that $|A||B|=|\Omega|$. Then the
following are equivalent:
\begin{enumerate}
\item for all $g\in G$, $|Ag\cap B|\geq 1$;
\item for all $g\in G$, $|Ag\cap B|=1$;
\item for all $g\in G$, $|Ag\cap B|\leq 1$.
\end{enumerate}
\end{lem}

The next observation will be used  in the 
proof of the next result: 
if $C$ is a clique and $A$ is an independent set of a graph, then
$|C\cap A|\leq 1$. 

\begin{lem}\label{CAequal}
Let $\Gamma$ be a vertex-transitive graph with $n$ vertices, 
let $C$ be a clique, and 
let $A$ be an independent set in $\Gamma$ such that $|C||A|=n$. Then
$\omega(\Gamma)=|C|$ and $\alpha(\Gamma)=|A|$. 
\end{lem}
\begin{proof}
It suffices to show that $\omega(\Gamma)=|C|$, as the 
equality $\alpha(\Gamma)=|A|$ will follow 
by considering the complement graph $\overline\Gamma$. 
Assume by contradiction that there is a clique $D$
in $\Gamma$ such that $|D|>|C|$. Assume without loss of generality that
$|D|=|C|+1$. Let $\beta\in D$ be a fixed element, and set 
$D_1=D\setminus\{\beta\}$. Then $D_1$ is a clique such that $|D_1|=|C|$. 
As 
$D_1g$ is a clique for all $g\in \Aut\Gamma$ and 
$A$ is an independence set, $|D_1g\cap A|\leq 1$. 
Since $|D_1||A|=n$, Lemma~\ref{lambda} gives that $|D_1g\cap A|= 1$ for 
all $g\in \Aut\Gamma$.
As $\Aut\Gamma$ is transitive on the vertices, there is 
$g_0\in\Aut\Gamma$ such that $\beta g_0\in A$. Since $|D_1g_0\cap A|=1$, we find that
$|Dg_0\cap A|=2$. Since $Dg_0$ is a clique and $A$ is an independent set, this
is a contradiction. Thus $\omega(\Gamma)=|C|$ as was required. 
\end{proof}




Using the fact that a generalized Paley graph is a normal Cayley graph, we 
obtain that the property 
$\omega(\Gamma)=\chi(\Gamma)$ that appears in Lemma~\ref{synchgraphs} 
follows from a 
weaker hypothesis.

\begin{lem}[Corollary 6.1.3 of \cite{gods}]\label{suitable-pseudo} 
If $\Gamma$ is a normal Cayley graph such that 
$\alpha(\Gamma)\omega(\Gamma)$ is equal to the number of vertices of $\Gamma$, 
then $\omega(\Gamma)=\chi(\Gamma)$.
\end{lem}

Lemma~\ref{suitable-pseudo} has an immediate consequence for generalized Paley 
graphs.

\begin{cor}\label{chiomega}
Set $\Gamma=\Gamma_{q,m}$ such that $m\mid q-1$ and $m\geq 2$. The following are equivalent.
\begin{enumerate}
\item $\omega(\Gamma)\alpha(\Gamma)=q$;
\item $\omega(\Gamma)=\chi(\Gamma)$;
\item $\omega(\overline\Gamma)\alpha(\overline\Gamma)=q$;
\item $\omega(\overline\Gamma)=\chi(\overline\Gamma)$.
\end{enumerate}
\end{cor}
\begin{proof}
Since $\Gamma$ is a normal Cayley graph (Corollary~\ref{PropertiesGenPaleyGraphs}), (1) implies (2) by 
Lemma~\ref{suitable-pseudo}. 
Suppose that (2) holds and let $C$ be a clique and $\P$ be a coloring
such that $|C|=|\P|=k$. Then $\P$ is a $G_{q,m}$-regular partition of the
vertex set such that the $G_{q,m}$-regularity is witnessed by $C$. Hence $\P$
is a uniform partition by Lemma~\ref{synchlemma}. Thus if $P\in \P$, then
$P$ is an independent set such that $|C||P|=q$ and so, 
by Lemma~\ref{CAequal}, $\omega(\Gamma)=k$
and $\alpha(\Gamma)=|P|=n/k$. Hence~(1) is valid. This shows that (1) and
(2) are equivalent. Similarly, we obtain that (3) and (4) are equivalent. 
Since 
$\omega(\Gamma)=\alpha(\overline\Gamma)$ and  $\alpha(\Gamma)=
\omega(\overline\Gamma)$, (1) is equivalent to (3) and the corollary is valid.
\end{proof}

In order to enrich our set of tools to study the clique and chromatic
numbers of generalized Paley graphs, we
use the $\vartheta$-function, a graph invariant defined by
Lov\'asz~\cite{lovasz}. 
The value of $\vartheta(\Gamma)$ for vertex-transitive and
edge-transitive graphs can be determined by the eigenvalues of the
adjacency matrix. As traditional, we let $\lambda_1^\Gamma$ and $\lambda_n^\Gamma$ denote the largest and the smallest eigenvalues of $\Gamma$, 
respectively.

\begin{prop}[\cite{lovasz}]\label{proptheta}
Let $\Gamma=(\Omega,E)$ be a graph with vertex set of size $n$. Then the following hold.
\begin{enumerate}
\item $\omega(\Gamma)\leq \vartheta(\overline\Gamma) \leq \chi(\Gamma)$.
\item If $\Gamma$ is vertex-transitive  then $\vartheta(\Gamma)\vartheta(\overline\Gamma)=n$.
\item If $\Gamma$ is regular then  $\lambda_1^\Gamma=\partial(\Gamma)$.
\item If $\Gamma$ is regular and edge-transitive then $\vartheta(\Gamma)=-n\lambda_n^\Gamma/(\lambda_1^\Gamma-\lambda_n^\Gamma)$.
\end{enumerate}
\end{prop}

Using Proposition~\ref{proptheta}, we obtain a necessary condition 
for the equality $\chi(\Gamma_{q,m})=\omega(\Gamma_{q,m})$ to hold 
in vertex-transitive and edge-transitive graphs.

\begin{teo}\label{thetaPaley}
Let $\Gamma$ be a vertex-transitive and edge-transitive graph with $n$ vertices such that \mbox{$\omega(\Gamma)=\chi(\Gamma)=k$}. Then $k-1$ divides $\partial(\Gamma)$ and  $\lambda_n^\Gamma = -\partial\Gamma/(k-1)$.
\end{teo}
\begin{proof}
Let $\Gamma$ be a graph satisfying the conditions of the theorem.  
Since $\omega(\Gamma)=\chi(\Gamma)=k$, the first statement of  
Proposition~\ref{proptheta} 
implies that $\omega(\Gamma)= \vartheta(\overline\Gamma) =\chi(\Gamma)$ and 
therefore $\vartheta(\overline\Gamma)=k$.  As $\Gamma$ is vertex-transitive, 
Proposition~\ref{proptheta}(2) gives
$\vartheta(\Gamma)=n/k$.
Also, by assumption, $\Gamma$ is regular and therefore we 
obtain, combining statements~(3) and~(4) of
Proposition~\ref{proptheta}, that
\begin{center}
$\vartheta(\Gamma)=\dfrac{n}{k}=\dfrac{-n\lambda_n^\Gamma}{\partial(\Gamma)-\lambda_n^\Gamma}$.
\end{center}
Thus
$\lambda_n^\Gamma (1-k)=\partial(\Gamma).$
Therefore $\lambda_n^\Gamma$ is a rational number. However, 
since $\lambda_n^\Gamma$ is an eigenvalue of the adjacency matrix of 
$\Gamma$, it follows, using the rational root test, that 
$\lambda_n^\Gamma$ is an integer.
Thus $(k-1)$ divides $\partial(\Gamma)$ and  \mbox{$\lambda_n^\Gamma= -\partial(\Gamma)/(k-1)$}.
\end{proof}

Next we state and prove the main result of this section.

\begin{teo}\label{main1}
Let $p$ be an odd prime and $n \in \N$. Set $q=p^n$. Let $m>1$ be such that 
$2m\mid q-1$ and set $\Gamma=\Gamma_{q,m}$.
Then the following hold.
\begin{enumerate}
\item  If $\omega(\Gamma)=\chi(\Gamma)$, then 
$\omega(\Gamma)=\chi(\Gamma)=p^{t}$, with some $t$ such that 
$t\mid n$.
\item If $n$ is a prime and
$\omega(\Gamma)=\chi(\Gamma)$, then
$m\mid p^{n-1}+p^{n-2}+\cdots+p+1$. Further, in this case, 
$\omega(\Gamma)=\chi(\Gamma)=p$. 
\item Suppose that $n$ is even. Then $\omega(\Gamma)=\chi(\Gamma)=p^{n/2}$ if 
and only if $m \mid (p^{n/2}+1)$. In particular, 
$\omega(\Gamma_{p^2,m})=\chi(\Gamma_{p^2,m})$ if and only if   $m\mid (p+1)$.
\item  If $n$ is odd and $m$ is even then $\omega(\Gamma)\neq\chi(\Gamma)$.
\end{enumerate}
\end{teo}
\begin{proof}
(1) Assume that $\omega(\Gamma)=\chi(\Gamma)=k$. 
By Corollary~\ref{PropertiesGenPaleyGraphs}(1), 
the group $G_{q,m}$ is a vertex-transitive
subgroup of $\Aut\Gamma$. If $\P$ is a coloring with $k$ colors and 
$C$ is a clique with $|C|=k$, then $Cg$ is a clique for all $g\in G_{q,m}$.
Therefore
$\P$ is a $G_{q,m}$-regular partition 
for the vertex set $\F_q$ of $\Gamma$ such that the $G_{q,m}$-regularity is witnessed
by $C$. Hence Lemma \ref{synchlemma}(1) gives 
that $k$ divides the number of vertices of $\Gamma$; 
that is, $k=p^t$, with $t<n$. 
Furthermore, using Theorem~\ref{thetaPaley}, we have that $p^t-1$ divides 
$\partial(\Gamma)=(p^n-1)/m$, and so $(p^t-1)\mid (p^n-1)$. 
Elementary argument shows that this implies that $t\mid n$. 

(2) Suppose that $n$ is a prime and that $\omega(\Gamma)=\chi(\Gamma)$. 
Then it follows from~(1) that $\omega(\Gamma)=\chi(\Gamma)=p$.
Therefore Theorem~\ref{thetaPaley} shows that $(p-1)\mid (p^n-1)/m$. 
As $(p-1)(p^{n-1}+p^{n-2}+\cdots+p+1)=p^n-1$, this implies the assertion.

(3) 
Assume that $n$ is even. If $\omega(\Gamma)=\chi(\Gamma)=p^{n/2}$ then we have, by Theorem~\ref{thetaPaley}, that 
$(p^{n/2}-1) \mid (p^n-1)/m$ which is equivalent to $m \mid (p^{n/2}+1)$.
Conversely, suppose that $m\mid (p^{n/2}+1)$. In this case 
$\F_{p^{n/2}}^*$ is a subgroup of $S_{p^n,m}$, and so the subfield
$\F_{p^{n/2}}$ is a clique in $\Gamma$.  This gives that 
$\omega(\Gamma)\geq p^{n/2}$ and also that 
$\alpha(\overline\Gamma)\geq p^{n/2}$ where $\overline\Gamma$ is the 
complement of $\Gamma$. On the other hand, by 
Corollary~\ref{PropertiesGenPaleyGraphs}, the graph $\Gamma$ 
can be embedded into $\overline\Gamma$, which shows that 
$\omega(\overline\Gamma)\geq p^{n/2}$.  By Lemma~\ref{CAequal},
 $\omega(\overline\Gamma)=\alpha(\overline\Gamma)=p^{n/2}$, 
and so $\omega(\Gamma)=\alpha(\Gamma)=p^{n/2}$. Thus
Lemma~\ref{suitable-pseudo}
implies that $\omega(\Gamma)=\chi(\Gamma)=p^{n/2}$. 

(4) Suppose now that  $n$ is odd and that $\omega(\Gamma)=\chi(\Gamma)=p^t$.  
As $(p^t-1)\mid(p^n-1)/m$, we find that
$m\mid (p^n-1)/(p^t-1)$ (which is an integer as $t\mid n$). As $n/t$ is odd, 
$(p^n-1)/(p^t-1)$ is a sum of $n/t$ $p$-th powers, and hence it is an odd number, 
which
implies that $m$ is an odd number.
Therefore~(4) is valid.
\end{proof}

\section{Synchronization of some afffine groups}\label{sec:groups}

We apply the results regarding the clique and chromatic
numbers of generalized Paley graphs in Section~\ref{sec:clique} to 
obtain information about the synchronization property of 
the 1-dimensional affine groups $G_{q,m}$ for odd $q$.  
Before stating the first result, recall
that a linear group acting on a vector space $V$ is said to be 
{\em imprimitive} if it preserves a direct sum decomposition
\begin{equation}\label{dirsum}
V=V_1\oplus\cdots\oplus V_s.
\end{equation} If this is not the case, then
the such a linear group is said to be {\em primitive}.
If $p$ is a prime, then the field $\F_{p^n}$
can be considered as an $n$-dimensional vector space over $\F_p$ and 
$Z_{p^n,m}$ can be considered as a subgroup of $\GL np$.

\begin{teo}\label{imprimitivesync}
If $Z_{q,m}$ is an imprimitive subgroup of $\GL np$, then $G_{q,m}$ is non-synchronizing.
\end{teo}
\begin{proof}
Suppose that $Z_{q,m}$ preserves the direct sum decomposition
as in~\eqref{dirsum}.
Set, for $i \in \{1,\ldots, s\}$ and $v\in V_i$,
\begin{center}
$P_{i,v}=\{v_1+\cdots+v_n \in V \mid v_i=v\}$
\end{center}
and consider the partition $\mathscr P_i=\{P_{i,v} \mid v \in V_i\}$. 
Then consider the set of partitions $\Sigma=\{\mathscr P_1, 
\dots, \mathscr P_s\}$. It follows from the imprimitivity of $Z_{q,m}$ and 
the definition of cartesian decompositions that 
$\Sigma$ is a cartesian decomposition which 
is preserved by $G_{q,m}$.
Thus, by Lemma \ref{cartdecomp}, we conclude that $G_{q,m}$ is a 
non-synchronizing group.
\end{proof}

\begin{teo}\label{main2}
Suppose that $q=p^n$ where $p$ is an odd prime and let $m$ be a divisor
of $q-1$. Set $\bar m$ as in equation~$\eqref{m1}$.
\begin{enumerate}
\item If $\omega(\Gamma_{q,\bar m})=\chi(\Gamma_{q,\bar m})$ then
$G_{q,m}$ is non-synchronizing. 
\item If $m\in\{2,3\}$, then $G_{q,m}$ is non-synchronizing if and only if 
$\omega(\Gamma_{q,m})=\chi(\Gamma_{q,m})$. 
\item Suppose that $n$ is a prime and that $m\in\{2,3\}$. If
$$
m\nmid p^{n-1}+p^{n-2}+\cdots+p+1, 
$$
then
$G_{q,m}$ is synchronizing.
\item If $n$ is even and $\gcd(m,p^{n/2}+1)\neq 1$, then 
$G_{q,m}$ is non-syn\-chro\-ni\-zing.
\item If $n$ is odd, then $G_{q,2}$ is synchronizing.
\item The group $G_{p^2,2}$ is always non-syn\-chro\-ni\-zing, 
while $G_{p^2,3}$ is non-syn\-chro\-ni\-zing
if and only if $m\mid  p+1$. 

\end{enumerate}
\end{teo}
\begin{proof}
(1) By Lemma~\ref{UndirectedOrbits}, 
the generalized Paley graph $\Gamma_{q,\bar m}$ is 
$G_{q,m}$-invariant. If $\omega(\Gamma_{q,\bar m})=\chi(\Gamma_{q,\bar m})$, 
then Lemma~\ref{synchgraphs} shows that $G_{q,m}$ is non-synchronizing.

(2) Suppose that $m\in\{2,3\}$. In this case $\bar m=m$. 
If $\Gamma$ is a $G_{q,m}$-invariant graph, then, as $G_{q,m}$ has 2 or 3 orbits
in $\Omega^{\{2\}}$, Lemma~\ref{UndirectedOrbits} shows that
$\Gamma$ is isomorphic either to $\Gamma_{q,m}$ or to the 
complement $\overline\Gamma_{q,m}$. On the other hand, 
Corollary~\ref{chiomega} shows that $\omega(\Gamma_{q,m})=\chi(\Gamma_{q,m})$ 
if
and only if $\omega(\overline\Gamma_{q,m})=\chi(\overline\Gamma_{q,m})$. 
This implies that $\omega(\Gamma)=\chi(\Gamma)$ if and only if 
$\omega(\Gamma_{q,m})=\chi(\Gamma_{q,m})$.

(3) If $n$ is prime, $m\in\{2,3\}$, and $m\nmid p^{n-1}+p^{n-2}+\cdots+p+1$, 
then Theorem~\ref{main1}(2), gives that 
$\omega(\Gamma_{q,m})\neq\chi(\Gamma_{q,m})$. Thus part~(1) shows that
$G_{q,m}$ is synchronizing.

(4) Suppose that $n$ is even and $d=\gcd(m,p^{n/2}+1)\neq 1$. 
Since $d \mid p^{n/2}+1$, 
Theorem~\ref{main1}(3) gives that $\omega(\Gamma_{q,d})=\chi(\Gamma_{q,d})$, 
and hence, by part~(1), $G_{q,d}$ is non-synchronizing. On the other hand, as 
$d \mid m$, we have that $G_{q,m}\leq G_{q,d}$, and hence
$G_{q,m}$ is non-synchronizing also.
%

(5) This follows combining part~(2) with Theorem~\ref{main1}(4).

(6) Part~(4) implies that $G_{p^2,2}$ is always non-synchronizing. By 
part~(2), $G_{p^2,3}$  is non-synchronizing if and only if 
$\omega(\Gamma_{p^2,3})=\chi(\Gamma_{p^2,3})$ which
happens if and only if $m\mid p+1$; see Theorem~\ref{main1}(3).

\end{proof}

\end{document}